\documentclass[fleqn,11pt,twoside]{article}

\usepackage{hyperref}

\usepackage{amsfonts, amsmath, amssymb, amsthm}

\usepackage{mathtools}

\usepackage{a4}
\usepackage{graphicx}

\usepackage{bbold}
\usepackage{bm}

 \usepackage{xcolor}

\makeatletter
\newcommand{\copyrightnote}[2]{{\renewcommand{\thefootnote}{}
 \footnotetext{\small\it
\begin{flushleft}
 \copyright \ #1   #2  
\end{flushleft}}}}

\newcommand{\Name}[1]{\begin{flushleft}
                       \LARGE \bf #1
                       \end{flushleft}\vspace{-3mm}}

\newcommand{\Author}[1]{\begin{flushleft}
                       \it #1 \end{flushleft}}

\newcommand{\Address}[1]{\begin{flushleft}
                       \it #1 \end{flushleft}}

\newcommand{\Date}[1]{\begin{flushleft}
                      \small  \it #1 \end{flushleft}}

\newcommand{\evenhead}{Author \ name}
\newcommand{\oddhead}{Article \ name}

\renewcommand{\@evenhead}{
\hspace*{-3pt}\raisebox{-15pt}[\headheight][0pt]{\vbox{\hbox to \textwidth
{\thepage \hfil \evenhead}\vskip4pt \hrule}}}
\renewcommand{\@oddhead}{
\hspace*{-3pt}\raisebox{-15pt}[\headheight][0pt]{\vbox{\hbox to \textwidth
{\oddhead \hfil \thepage}\vskip4pt\hrule}}}
\renewcommand{\@evenfoot}{}
\renewcommand{\@oddfoot}{}

\long\def\@makecaption#1#2{%
  \vskip\abovecaptionskip
  \sbox\@tempboxa{\small \textbf{#1.}\ \ #2}%
  \ifdim \wd\@tempboxa >\hsize
    {\small \textbf{#1.}\ \ #2}\par
  \else
    \global \@minipagefalse
    \hb@xt@\hsize{\hfil\box\@tempboxa\hfil}%
  \fi
  \vskip\belowcaptionskip}

\newcommand{\JNMPnumberwithin}[3][\arabic]{%
  \@ifundefined{c@#2}{\@nocounterr{#2}}{%
    \@ifundefined{c@#3}{\@nocnterr{#3}}{%
      \@addtoreset{#2}{#3}%
      \@xp\xdef\csname the#2\endcsname{%
        \@xp\@nx\csname the#3\endcsname .\@nx#1{#2}}}}%
}

\renewenvironment{proof}[1][\proofname]{\par
  \normalfont
  \topsep6\p@\@plus6\p@ \trivlist
  \item[\hskip\labelsep\textbf{%
    #1\@addpunct{.}}]\ignorespaces
}{%
  \qed\endtrivlist
}

\newcommand{\resetfootnoterule} {
  \renewcommand\footnoterule{%
  \kern-3\p@
  \hrule\@width.4\columnwidth
  \kern2.6\p@}
}

\renewcommand{\footnoterule}{}

\makeatother

 \definecolor{Refkey}{RGB}{255,127,0}
 \definecolor{Labelkey}{RGB}{127,0,255}
 \makeatletter 
  \def\SK@refcolor{\color{Refkey}}
  \def\SK@labelcolor{\color{Labelkey}}
 \makeatother

  \definecolor{mdg}{RGB}{0,177,0} 
  \definecolor{mdb}{RGB}{0,0,191}
  \definecolor{mddb}{RGB}{0,0,91}
  \definecolor{mdy}{RGB}{255,69,0} 
  \definecolor{gray}{RGB}{99,99,99} 
  \definecolor{darkgreen}{RGB}{0,128,0} 
  \definecolor{darkblue}{RGB}{0,0,128}

\DeclareMathOperator{\diag}{diag}

\newtheorem{proposition}{Proposition}

\theoremstyle{definition}
\newtheorem{definition}{Definition}

\theoremstyle{remark}
\newtheorem{remark}{Remark}

\begin{document}

\renewcommand{\evenhead}{ {\LARGE\textcolor{blue!10!black!40!green}{{\sf \ \ \ ]ocnmp[}}}\strut\hfill 
I G Korepanov
}
\renewcommand{\oddhead}{ {\LARGE\textcolor{blue!10!black!40!green}{{\sf ]ocnmp[}}}\ \ \ \ \  
Self-similarity on 4d cubic lattice
}

\thispagestyle{empty}
\newcommand{\FistPageHead}[3]{
\begin{flushleft}
\raisebox{8mm}[0pt][0pt]
{\footnotesize \sf
\parbox{150mm}{{\textcolor{blue!10!black!40!green}{{\bf Open Communications in Nonlinear Mathematical Physics}}}
\ \ {Special Issue: Hietarinta}, 2025/2026\\[0.1cm]
\strut\hfill pp
#2\hfill {\sc #3}}}\vspace{-13mm}
\end{flushleft}}

\FistPageHead{1}{\pageref{firstpage}--\pageref{lastpage}}{ \ \ }

\strut\hfill

\strut\hfill

\copyrightnote{The author. Distributed under a Creative Commons Attribution 4.0 International License}

\begin{center}

{\bf {\large A Special OCNMP Issue in Honour of Jarmo Hietarinta\\ 

\smallskip

on the occasion of his 80th Birthday}}
\end{center}

\smallskip

\Name{Self-similarity on 4d cubic lattice}

\Author{Igor G. Korepanov}

\Address{Dept. of Computer Sciences and Applied Mathematics, 
Moscow Aviation Institute, 
4~Volokolamskoe shosse, 
125993 Moscow, 
Russian Federation}

\Date{Received April 4, 2025; Accepted May 26, 2025}

\setcounter{equation}{0}

\begin{abstract}

\noindent 
A phenomenon of ``algebraic self-similarity'' on 3d cubic lattice, providing what can be called an algebraic analogue of Kadanoff--Wilson theory, is shown to possess a 4d version as well. Namely, if there is a $4\times 4$ matrix~$A$ whose entries are indeterminates over the field~$\mathbb F_2$, then the $2\times 2\times 2\times 2$ block made of sixteen copies of~$A$ reveals the existence of four direct ``block spin'' summands corresponding to the same matrix~$A$. Moreover, these summands can be written out in quite an elegant way. Somewhat strikingly, if the entries of~$A$ are just zeros and ones---elements of~$\mathbb F_2$---then there are examples where two more ``block spins'' split out, and this time with different~$A$'s.
\end{abstract}

\label{firstpage}


\sloppy

\section{Introduction}\label{s:i}

This paper is significantly inspired by Hietarinta's work~\cite{Hietarinta}.
We develop some algebra, dealing with four-dimen\-sional cubic lattices having a ``spin'' at each of its edges, and a copy of a $4\times 4$ matrix
\begin{equation}\label{A}
A = \begin{pmatrix} a_{11} & a_{12} & a_{13} & a_{14} \\ a_{21} & a_{22} & a_{23} & a_{24} \\ a_{31} & a_{32} & a_{33} & a_{34} \\ a_{41} & a_{42} & a_{43} & a_{44} \end{pmatrix}
\end{equation}
in each of its vertices that governs the permitted configurations of the adjacent spins. Namely, we choose positive directions for all four coordinate axes and assume that all edges are oriented according to these directions; this implies that each vertex has four incoming, or \emph{input}, edges, and for outgoing, or \emph{output}, ones. Denoting the input spins ( = spins on incoming edges) as~$x_i$, \ $i=1,\ldots,4$, and output spins similarly as~$y_i$, we require that
\begin{equation}\label{v}
\begin{pmatrix} x_1 & x_2 & x_3 & x_4 \end{pmatrix} A = \begin{pmatrix} y_1 & y_2 & y_3 & y_4 \end{pmatrix}
\end{equation}
at each vertex, see Figure~\ref{f:A}.
\begin{figure}
 \centering
 \includegraphics[scale=1]{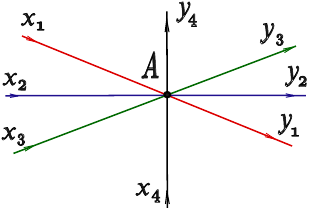}
 \caption{One vertex of a 4d lattice and its adjacent spins. We fix, with this picture, the way we depict the four coordinate axes---their directions on the 2d plane, and their colors ($x_j$ and~$y_j$ are of course on a line parallel to the $j$-th axis)}
 \label{f:A}
\end{figure}

Both our spins and matrix entries in~\eqref{A} are supposed to belong to a \emph{field of a finite characteristic}. To be exact, in this paper, it will be either the smallest Galois field~$\mathbb F_2$, or its transcendental extension by (as many as needed) independent indeterminates.

\begin{remark}
It can be seen already from~\eqref{v} that we prefer to use \emph{row} vectors that are, accordingly, multiplied by matrices from the \emph{right}. One of the reasons is that the GAP computer algebra system~\cite{GAP}, which we used extensively, has the same clear preference.
\end{remark}

This construction may be considered as a simplified model of 4d lattice statistical physics. Earlier, similar models in three dimensions were considered in~\cite{block-spin} and~\cite{prodolzhenie}. In~\cite{block-spin}, a remarkable property was discovered that we call \emph{algebraic self-similarity}: a cubic block of matrices~$A$ decomposes into a \emph{direct sum} of new vertices, each bearing a matrix related in a simple way to the initial~$A$. In~\cite{prodolzhenie}, it was shown that this makes it possible to calculate some quite nontrivial spin correlation functions at least for a ``combinatorial'' model, where all permitted configurations have equal probabilities.

In theoretical physics, self-similarity on lattices is studied usually within the framework of \emph{Kadanoff--Wilson theory}~\cite{Kadanoff,Wilson}, and is applied to the theory of phase transitions. The reader can find more on that subject, e.g., in~\cite[Chapter~4]{Gitterman-Halpern}. Our theory can thus be called an \emph{algebraic analogue} of Kadanoff--Wilson theory.

Concerning our motivation to study the \emph{four}-dimen\-sional case, we think it is interesting already because it reveals, as we will show, new and interesting algebraic structures. It must also be kept in mind that a reduction to the physically more realistic 3d case may be found.

\begin{remark}
Some way of making ``four dimensions from three dimensions'' has been already proposed in~\cite[Section~VI]{block-spin}. It looks, however, too restrictive, because of the very special boundary conditions that we had to impose there. It is certainly desirable to have more freedom in boundary conditions if we want, for instance, to define and calculate interesting correlation functions, like we did in~\cite{prodolzhenie}. Moreover, block spins in that approach correspond to (long) strings made of $3\times 3$ matrices; we think that our current approach where a block spin corresponds to a single $4\times 4$ matrix is more elegant, and more promising both from the viewpoint of studying its algebraic structure, and from the viewpoint of possible applications.
\end{remark}

\section[General case, with indeterminates over~\texorpdfstring{$\mathbb F_2$}{F2}]{General case, with indeterminates over~$\bm{\mathbb F_2}$}\label{s:g}

In this section, we take $a_{ij}$ in~\eqref{A} to be \emph{indeterminates} over the Galois field~$\mathbb F_2$. Accordingly, spins belong now to the field~$F$ of rational functions of all~$a_{ij}$, and linear spaces are considered over~$F$.

We make then a $2\times 2\times 2\times 2$ block~$\mathbf B$ of copies of matrix~$A$, as shown in Figure~\ref{f:B}.
\begin{figure}
 \centering
 \includegraphics[scale=1]{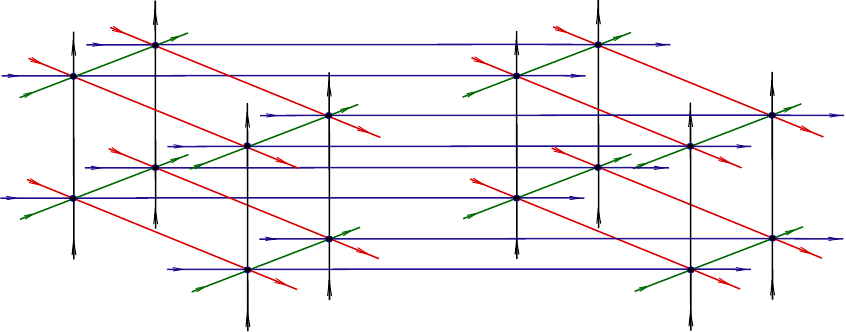}
 \caption{A $2\times 2\times 2\times 2$ block~$\mathbf B$. It is understood that a copy of matrix~$A$ is at each of 16 vertices, marked by small black filled circles}
 \label{f:B}
\end{figure}
In each of the four coordinate directions, numbered below by $j=1,\ldots,4$, there are eight input or output spins. We write them (either input or output spins) together as an 8-row vector, and call the linear space formed by them~$\mathbf V_j$, for each direction~$j$. The whole block~$\mathbf B$ acts in the direct sum 
\begin{equation}\label{V}
\mathbf V = \mathbf V_1 \oplus \mathbf V_2 \oplus \mathbf V_3 \oplus \mathbf V_4 \,,
\end{equation}
accordingly, we write its elements as 32-rows composed of the mentioned four 8-rows taken in the natural order.

\begin{definition}\label{d:thsp}
Spaces~$\mathbf V_j$ of 8-rows containing either input or output spins belonging to edges parallel to any of the coordinate directions $j=1,\ldots,4$, as described above, are called below \emph{thick spaces}.
\end{definition}

The order of entries in each 8-row, or, which is the same, the order of eight edges in Figure~\ref{f:B} going in a fixed direction, is chosen in the following natural way. Assume that there is a 4d coordinate system whose axes are situated in such way that the coordinates of all vertices with matrices~$A$ in Figure~\ref{f:B} belong to the Cartesian product $\{0,1\}\times \{0,1\}\times \{0,1\}\times \{0,1\}$ (that is, any coordinate of any vertex is either~$0$ or~$1$), and the directions of the axes coincide with the directions of the edges in that figure. 

\begin{definition}\label{d:e}
We call \emph{coordinates of an edge} the \emph{three} coordinates of its any point in the mentioned coordinate system, excluding from its four coordinates one belonging to the axis parallel to the edge in question.
\end{definition}

Returning to the question of ordering the eight parallel edges, we use the \emph{lexicographic} ordering of their coordinates.

\smallskip

There might be some confusion in terminology in this paper, because the word ``block'' can have, in our context, two meanings. First, $\mathbf B$ is itself a block of copies of matrix~$A$, in the same sense as in the Kadanoff--Wilson theory. Second, as $\mathbf B$ acts in a direct sum of four linear spaces, $\mathbf B$ may be represented by a \emph{block matrix} consisting of ``blocks'' acting from a given $\mathbf V_i$ into a given $\mathbf V_j$. We reserve here the word ``block'' for its first sense, while in the second case, we rename blocks into ``thick matrix entries'', according to the following definition.

\begin{definition}\label{d:thme}
The submatrix of~$\mathbf B$ consisting of rows corresponding to~$\mathbf V_i$ and columns corresponding to~$\mathbf V_j$  will be denoted~$\mathbf b_{jk}$ and called \emph{thick matrix entry} of~$\mathbf B$.
\end{definition}

It turns out that there are vectors~$\mathsf e_i^{(j)} \in \mathbf V_j$, \emph{four} such linearly independent vectors ($i=1,\ldots,4$) in each~$\mathbf V_j$, such that block~$\mathbf B$ acts on them like matrix~$\tilde A$ obtained from~$A$ by the Frobenius transform
\begin{equation}\label{F}
a_{jk} \mapsto \tilde a_{jk} = a_{jk}^2
\end{equation}
of all the entries:
\begin{equation}\label{bjk}
\mathsf e_i^{(j)} \mathbf b_{jk} = a_{jk}^2 \mathsf e_i^{(k)}.
\end{equation}
In other words, a 32-row $\lambda _1 \mathsf e_i^{(1)} + \lambda _2 \mathsf e_i^{(2)} + \lambda _3 \mathsf e_i^{(3)} + \lambda _4 \mathsf e_i^{(4)}$ is taken by (the \emph{right} action of) $\mathbf B$ into $\mu _1 \mathsf e_i^{(1)} + \mu _2 \mathsf e_i^{(2)} + \mu _3 \mathsf e_i^{(3)} + \mu _4 \mathsf e_i^{(4)}$, where
\begin{equation}\label{lm}
\begin{pmatrix} \mu _1 & \mu _2 & \mu _3 & \mu _4 \end{pmatrix} = \begin{pmatrix} \lambda _1 & \lambda _2 & \lambda _3 & \lambda _4 \end{pmatrix} \tilde A \,.
\end{equation}
The actual form of vectors~$\mathsf e_i^{(j)}$ will be given soon in Proposition~\ref{p:e}. Here we note that, although relations~\eqref{bjk} are obviously conserved if we change these vectors to their linear combinations in the following way:
\begin{equation*}
\begin{pmatrix} \mathsf e_1^{(j)} \\ \mathsf e_2^{(j)} \\ \mathsf e_3^{(j)} \\ \mathsf e_4^{(j)} \end{pmatrix} \mapsto \begin{pmatrix} \mathsf m_1^{(j)} \\ \mathsf m_2^{(j)} \\ \mathsf m_3^{(j)} \\ \mathsf m_4^{(j)} \end{pmatrix} = M \begin{pmatrix} \mathsf e_1^{(j)} \\ \mathsf e_2^{(j)} \\ \mathsf e_3^{(j)} \\ \mathsf e_4^{(j)} \end{pmatrix} ,
\end{equation*}
where $M$ is a nondegenerate $4\times 4$ matrix (the same for all~$j$), the mentioned Proposition~\ref{p:e} clearly shows that there is a preferred specific choice of~$\mathsf e_i^{(j)}$.

\begin{remark}
Index~$i$ in~$\mathsf e_i^{(j)}$ has, apparently, a nature different from~$j$. Namely, $j$ numbers our thick spaces~$\mathbf V_j$, while $i$ numbers four vectors within each of them. This makes even more surprising Proposition~\ref{p:e} below, because it implies, as one can easily see, that changing~$A$ to its transpose~$A^{\mathrm T}$ leads to interchange $i\leftrightarrow j$.
\end{remark}

\begin{definition}\label{d:rp}
We put a \emph{representing polynomial} depending on four indeterminates $u_j$, \ $j=1,\ldots,4$, in correspondence to a vector in each of the spaces~$\mathbf V_j$ according to the following principle: for a vector $\begin{pmatrix} x_1 & \ldots & x_8 \end{pmatrix} \in \mathbf V_j$, take all~$u_i$ \emph{except}~$u_j$, not changing their order, and denote them as $v_1, v_2, v_3$ (for instance, if $j=1$, then $v_1=u_2$, $v_2=u_3$, and~$v_3=u_4$). The polynomial is, by definition,
\begin{equation}\label{rp}
 x_1 v_1 v_2 v_3 + x_2 v_1 v_2 + x_3 v_1 v_3 + x_4 v_1 + x_5 v_2 v_3 + x_6 v_2 + x_7 v_3 + x_8 .
\end{equation}
\end{definition}

Recall that an edge in Figure~\ref{f:B} parallel to $j$-th coordinate axis (with a fixed $j\in \{1,2,3,4\}$) has, according to Definition~\ref{d:e}, three coordinates, each being either~$0$ or~$1$, and our spins~$x_k$---components of vector $\begin{pmatrix}x_1 & \ldots & x_8\end{pmatrix} \in \mathbf V_j$---correspond to these edges taken in the lexicographic order. This means that indeterminates~$u_l$ in~\eqref{rp} correspond thus to the four coordinate axes in the following simple way: $x_k$ is multiplied by~$u_l$ if the $l$-th coordinate of the corresponding edge is~$0$, and not multiplied if~$1$. See these multipliers in Figure~\ref{f:C} for the example of the second thick space~$\mathbf V_2$.
\begin{figure}
 \centering
 \includegraphics[scale=1.5]{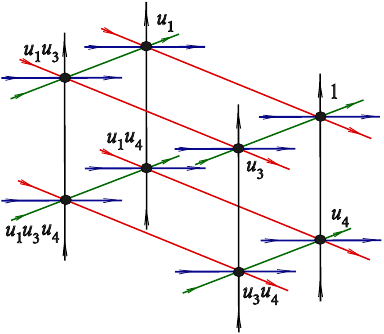}
 \caption{Products of variables~$u_i$ for the example of edges parallel to the 2nd axis (remember that these are colored blue). They are multiplied by spins standing at their corresponding blue edges, according to~\eqref{rp}. Here we show only one 3-face of the 4d cube perpendicular to the 2nd axis; it may be either input or output}
 \label{f:C}
\end{figure}

\begin{proposition}\label{p:e}
Vectors~$\mathsf e_i^{(j)}$ satisfying~\eqref{bjk} can be obtained from the representing polynomials equal to the \emph{cofactors} of the following matrix:
\begin{equation}\label{C}
C = A + \diag(u_1,u_2,u_3,u_4) = \begin{pmatrix} a_{11}+u_1 & a_{12} & a_{13} & a_{14} \\ a_{21} & a_{22}+u_2 & a_{23} & a_{24} \\ a_{31} & a_{32} & a_{33}+u_3 & a_{34} \\ a_{41} & a_{42} & a_{43} & a_{44}+u_4 \end{pmatrix} , 
\end{equation}
that is, to obtain the polynomial representing~$\mathsf e_i^{(j)}$, delete the $i$-th row and the $j$-th column from~$C$ and take the determinant of the resulting submatrix.
\end{proposition}

\begin{proof}
This somewhat mysterious statement can be proved, at this stage, only by a direct calculation. A GAP program for this calculation in presented in Appendix~\ref{a:GAP}.
\end{proof}

Recall that, according to what we said in the beginning of this section, $\mathsf e_i^{(j)}$ can be considered as an element of linear space~$\mathbf V_j$, as well as of $\mathbf V = \mathbf V_1 \oplus \mathbf V_2 \oplus \mathbf V_3 \oplus \mathbf V_4$. In the first case, we write it as an 8-row, while in the second---as a 32-row, adding zeros at the places corresponding to other spaces.

\begin{proposition}\label{p:d}
Linear subspace $\mathbf W \subset \mathbf V$ spanned by all vectors~$\mathsf e_i^{(j)}$ is a \emph{direct} summand with respect to the action of~$\mathbf B$. That is, there exists another linear space~$\mathbf W'$, also invariant under~$\mathbf B$ and such that $\mathbf W \oplus \mathbf W' = \mathbf V$.
\end{proposition}

\begin{proof}
This can be proved based on considering the transposed matrix~$\mathbf B^{\mathrm T}$ for the block~$\mathbf B$. As one can see, $\mathbf B^{\mathrm T}$ is obtained from~$A^{\mathrm T}$ in the same way as $\mathbf B$ from~$A$, except that the \emph{order of vertices along each coordinate axis is reversed}. That is, each coordinate of each vertex must be changed: $0\mapsto 1$ and~$1\mapsto 0$.

We can thus introduce invariant space~$\mathbf W_{\mathrm T}$ for~$\mathbf B^{\mathrm T}$ in the same way as we introduced space~$\mathbf W$ for~$\mathbf B$, and then take $\mathbf W' = (\mathbf W_{\mathrm T})^{\perp}$, that is, the space of rows orthogonal to~$\mathbf W_{\mathrm T}$ in the sense of component-wise product. Then it is checked by a direct calculation that, first, $\mathbf W \cap \mathbf W' = \{0\}$, and second, $\mathbf W \oplus \mathbf W'$ gives the whole~$\mathbf V$.
\end{proof}

\begin{remark}\label{r:W}
Linear space $\mathbf W$ is certainly a direct sum
\begin{equation}\label{W}
\mathbf W = \mathbf W_1 \oplus \mathbf W_2 \oplus \mathbf W_3 \oplus \mathbf W_4
\end{equation}
of its intersections $\mathbf W_j = \mathbf W \cap \mathbf V_j$ with our thick spaces~$\mathbf V_j$. Also, each~$\mathbf W_j$ is of course spanned be vectors~$\mathsf e_i^{(j)}$, \ $i=1,\ldots,4$. We will be using a particular case of spaces~$\mathbf W_j$ below in Proposition~\ref{p:fg}.
\end{remark}

\begin{remark}
Linear space $\mathbf W'$ is also a direct sum of its intersections~$\mathbf W'_j$ with~$\mathbf V_j$. The problem is that the action of~$\mathbf B$ within~$\mathbf W'$ looks much more complicated than within~$\mathbf W$.
\end{remark}

\section[Additional phenomena with just elements of~\texorpdfstring{$\mathbb F_2$}{F2}]{Additional phenomena with just elements of~$\bm{\mathbb F_2}$}\label{s:a}

When matrix entries of~$A$ are just zeros and ones, some new and unexplained (as yet) phenomena may occur. Namely, in addition to the four ``block spin'' direct summands in the block~$B$, there appear two more of them, but with \emph{different}~$A$'s!

Choose, for instance,
\begin{equation}\label{k2}
 A = 
 \begin{pmatrix} 1 & 1 & 1 & 1 \\
                 0 & 1 & 1 & 0 \\
                 0 & 0 & 1 & 1 \\
                 1 & 0 & 0 & 0  
 \end{pmatrix}
\end{equation}

\begin{proposition}\label{p:fg}
For block~$\mathbf B$ made of matrices~\eqref{k2} according to Figure~\ref{f:B}, each space~$\mathbf V_j$ decomposes in the following direct sum:
\begin{equation}\label{WFGH}
\mathbf V_j = \mathbf W_j \oplus \mathbf F_j \oplus \mathbf G_j \oplus \mathbf H_j ,
\end{equation}
in such way that their direct sums $\mathbf W$, $\mathbf F$, $\mathbf G$ and~$\mathbf H$ are invariant under~$\mathbf B$ (here $\mathbf W$ is defined as in~\eqref{W}, and the other three spaces similarly).

Spaces~$\mathbf W_j$ in~\eqref{WFGH} are the same as in Section~\ref{s:g} (see Remark~\ref{r:W}), except that $a_{ij}$ are no longer indeterminates but zeros and ones according to~\eqref{k2}. Block~$\mathbf B$ acts within~$\mathbf W$ also in the same way---decomposes into a direct sum of four matrices~$A$~\eqref{k2}.

Each space~$\mathbf F_j$ is one-dimen\-sional, and spanned by the vector~$\mathsf f^{(j)}$ from the following list:
\begin{equation}\label{f}
\begin{aligned}
\mathsf f^{(1)} &= \begin{pmatrix} 0 & 0 & 1 & 0 & 0 & 0 & 1 & 0 \end{pmatrix} , \\
\mathsf f^{(2)} &= \begin{pmatrix} 0 & 0 & 0 & 0 & 0 & 1 & 0 & 1 \end{pmatrix} , \\
\mathsf f^{(3)} &= \begin{pmatrix} 1 & 0 & 0 & 0 & 1 & 1 & 0 & 0 \end{pmatrix} , \\
\mathsf f^{(4)} &= \begin{pmatrix} 0 & 0 & 0 & 0 & 0 & 1 & 0 & 1 \end{pmatrix} .
\end{aligned}
\end{equation}
Block~$\mathbf B$ acts within~$\mathbf F$ as the following matrix:
\begin{equation}\label{k2-1}
 A_1 = 
 \begin{pmatrix} 1 & 0 & 1 & 1 \\
                 0 & 1 & 0 & 1 \\
                 0 & 1 & 1 & 1 \\
                 1 & 0 & 0 & 0  
 \end{pmatrix} ,
\end{equation}
Namely, we have (compare~\eqref{bjk})
\begin{equation*}
\mathsf f^{(j)} \mathbf b_{jk} = (a_1)_{jk} \mathsf f^{(k)} .
\end{equation*}

Similarly, each space~$\mathbf G_j$ is one-dimen\-sional, and spanned by the vector~$\mathsf g^{(j)}$ from the following list:
\begin{equation}\label{g}
\begin{aligned}
\mathsf g^{(1)} &= \begin{pmatrix} 0 & 1 & 0 & 1 & 0 & 1 & 0 & 1 \end{pmatrix} , \\
\mathsf g^{(2)} &= \begin{pmatrix} 0 & 0 & 0 & 0 & 0 & 0 & 1 & 0 \end{pmatrix} , \\
\mathsf g^{(3)} &= \begin{pmatrix} 0 & 0 & 0 & 0 & 1 & 0 & 1 & 0 \end{pmatrix} , \\
\mathsf g^{(4)} &= \begin{pmatrix} 1 & 0 & 1 & 1 & 1 & 0 & 1 & 1 \end{pmatrix} . 
\end{aligned}
\end{equation}
Block~$\mathbf B$ acts within~$\mathbf G$ as the transpose of~\eqref{k2-1}, that is, matrix
\begin{equation}\label{k2-2}
 A_2 = A_1^{\mathrm T} .
\end{equation}
Namely, we have
\begin{equation*}
\mathsf g^{(j)} \mathbf b_{jk} = (a_2)_{jk} \mathsf g^{(k)} .
\end{equation*}

The action of~$\mathbf B$ in~$\mathbf H$ is described in Appendix~\ref{a:r}.
\end{proposition}

\begin{proof}
Again direct calculation.
\end{proof}

\begin{remark}
The phenomena described in Proposition~\ref{p:fg} were found with the help of computer algebra package Singular~\cite{Singular}. The interested reader can \emph{check} them, however, using GAP, by slightly modifying the program in Appendix~\ref{a:GAP}.
\end{remark}

\section{Discussion}\label{s:d}

It is of course interesting what will happen if we continue by making blocks of blocks. At our first step, there appeared, from our initial matrix~\eqref{k2}, four copies of itself, one matrix~\eqref{k2-1}, one matrix~\eqref{k2-2}, and something more complicated that we describe in Appendix~\ref{a:r}. The natural question is what appears when we make a $2\times 2\times 2\times 2$ block out of matrices~$A_1$ (then we will see that also for~$A_2$ by means of transposing).

The answer is that we obtain four copies of~$A_1$ (as we could expect) together with one \emph{initial}~$A$ and one its transpose~$A^{\mathrm T}$. 

Such phenomena give us a hope that some interesting combinatorial correlations can be calculated, in the spirit of work~\cite{prodolzhenie}.

The main reason for our interest in this subject is, however, the fascinating algebra, of which, it looks like only a small part has been already revealed.

\appendix

\section[A GAP program to check the properties of vectors \texorpdfstring{$\mathsf e_i^{(j)}$}{ei(j)}]{A GAP program to check the properties of vectors $\bm{\mathsf e_i^{(j)}}$}\label{a:GAP}

A working GAP code can be copied directly from here below:

\begin{verbatim}
char := 2;  # the characteristic of the field
n := 4;     # the dimension of the space

field := GF(char); # Galois field of two elements

# Now we introduce matrix A whose entries A[i][j] are 
                          # denoted a_{ij} in the text
A := NullMat( n, n, field );
 for i in [1..n] do
  for j in [1..n] do
   A[i][j] := X( field, Concatenation( "A[", String(i), 
                              "][", String(j), "]" ) );
  od;
 od;

# Matrix A is 4*4, but when we make a block of 16 copies of A,
# we must extend it to its direct sum with an identity matrix.
# How this is done, depends on the coordinates of the vertex 
# where this A is placed.
# Here we assume that each of the n coordinates takes values 
# in [0..char-1]. That is, either 0 or 1 in our case
big_A := function( A, coords )
 local seq, m, AA;
 seq := Cartesian( List( [1..n-1], x-> [0..char-1] ) );
 m := function( i )
  local without_i;
  without_i := ShallowCopy( coords );
  Remove( without_i, i ); # removing the i-th element from the 
                          # copy of coords
             # we thus obtain the coordinates of 
             # the i-th edge going through the considered vertex
  return (i-1) * char^(n-1) + Position( seq, without_i ); 
     # and here we number all edges in the following way:
     # first go the edges parallel to the first axis, 
     # then to the second, etc., while parallel edges go in the 
     # lexicographic order of their coordinates
 end;
 AA := IdentityMat( n * char^(n-1), field );
 for i in [1..n] do
  for j in [1..n] do
   AA[m(i)][m(j)] := A[i][j];
  od;
 od;
 return AA * One(A[1][1]); # GAP requires that all matrix entries
                                       # belong to the same field
      # this time -- to the field containing our indeterminates
end;

f := function(A) # product of the copies of a matrix A 
                 # over all vertices in the block
                 # starting from the leftmost multiplier 
                 # corresponding to vertex [0, ..., 0]
 local g, c;
 g := function( i ) # this function within function f gives 
                    # all vertices whose sum of coordinates is i
  local all_vertices;
  all_vertices := Cartesian( List( [1..n], x-> [0..char-1] ) );
  return Filtered( all_vertices, x -> Sum(x) = i );
 end;
 c := Concatenation( List( [0 .. n*(char-1)], i -> g(i) ) );
 return Product( c, coords -> big_A( A, coords ) );
end;

F := f(A); # the actual product

# Thick matrix entries of F:
B := function( i, j )
 local r;
 r := char^(n-1);
 return F{[(i-1)*r+1 .. i*r]}{[(j-1)*r+1 .. j*r]};
end;


# Four more indeterminates:
u := List( [1..4], i -> X( field, 
                    Concatenation( "u[", String(i), "]" ) ) );

# 0 and 1 belonging to the required field
zero := Zero(u[1]);
one := One(u[1]);

# Making matrix C:
C := StructuralCopy( A );
C[1][1] := C[1][1] + u[1];
C[2][2] := C[2][2] + u[2];
C[3][3] := C[3][3] + u[3];
C[4][4] := C[4][4] + u[4];

# Making vectors e_i^(j):
e := function( j, i )   # j is the number of the space V_i
 local without_j, v, without_i, H, d, l;
 without_j := Difference( [1..4], [j] );
 v := List( without_j, k -> u[k] );
 without_i := Difference( [1..4], [i] );
 H := C{without_i}{without_j};
 d := Determinant( H );
 l := [ ];
 l[1] := Value( Derivative( Derivative( Derivative( d, 
                v[1] ), v[2] ), v[3] ), v, [zero,zero,zero] );
 l[2] := Value( Derivative( Derivative( d, v[1] ), v[2] ), v,
                                            [zero,zero,zero] );
 l[3] := Value( Derivative( Derivative( d, v[1] ), v[3] ), v, 
                                            [zero,zero,zero] );
 l[4] := Value( Derivative( d, v[1] ), v, [zero,zero,zero] );
 l[5] := Value( Derivative( Derivative( d, v[2] ), v[3] ), v, 
                                            [zero,zero,zero] );
 l[6] := Value( Derivative( d, v[2] ), v, [zero,zero,zero] );
 l[7] := Value( Derivative( d, v[3] ), v, [zero,zero,zero] ); 
 l[8] := Value( d, v, [zero,zero,zero] );
 l := l * one;
 return l;
end;

# It looks convenient to unite the four row vectors e_i^(j), 
# for a fixed j, in a matrix:
mat := function( j )
 return [ e(j,1), e(j,2), e(j,3), e(j,4) ];
end;

# Then, what we must check is written in GAP as follows:
check_Boolean_value := ForAll( [1..4], j -> ForAll( [1..4], k -> 
                mat(j) * B(j,k) = mat(k) * A[j][k]^2 ) );;

Print( check_Boolean_value, "\n" );
\end{verbatim}

\section{The more complicated direct summand}\label{a:r}

According to Proposition~\ref{p:fg}, block~$\mathbf B$ made of matrices~\eqref{k2} decomposes into the direct sum of four initial matrices~$A$, one matrix~$A_1$~\eqref{k2-1}, one matrix~$A_2=A_1^{\mathrm T}$~\eqref{k2-1}, and something in the remaining space~$\mathbf H$---the direct sum of four two-dimen\-sional spaces~$\mathbf H_j$, \ $j=1,\ldots,4$. Here is how~$\mathbf H$ and the action of~$\mathbf B$ within it can be described.

Space~$\mathbf H_j$, for $j=1,\ldots,4$, is spanned by two vectors $\mathsf h_1^{(j)}, \mathsf h_2^{(j)} \in \mathbf V_j$, which are as follows:
\begin{equation}\label{h}
 \begin{aligned}
 \begin{pmatrix} \mathsf h_1^{(1)} \\ \mathsf h_2^{(1)} \end{pmatrix} = 
 \begin{pmatrix}
   0 & 1 & 0 & 1 & 0 & 0 & 0 & 0  \\
       0 & 0 & 0 & 1 & 0 & 0 & 0 & 1 
 \end{pmatrix} , \\
 \begin{pmatrix} \mathsf h_1^{(2)} \\ \mathsf h_2^{(2)} \end{pmatrix} = 
 \begin{pmatrix}
    0 & 1 & 0 & 1 & 0 & 1 & 0 & 1  \\
       0 & 0 & 0 & 0 & 1 & 0 & 1 & 0  
 \end{pmatrix} , \\
 \begin{pmatrix} \mathsf h_1^{(3)} \\ \mathsf h_2^{(3)} \end{pmatrix} = 
 \begin{pmatrix}
    0 & 0 & 0 & 0 & 1 & 0 & 0 & 0  \\
       0 & 1 & 0 & 0 & 0 & 1 & 0 & 0   
 \end{pmatrix} , \\
 \begin{pmatrix} \mathsf h_1^{(4)} \\ \mathsf h_2^{(4)} \end{pmatrix} = 
 \begin{pmatrix}
    1 & 0 & 0 & 1 & 1 & 0 & 0 & 1  \\
       0 & 1 & 0 & 1 & 0 & 1 & 0 & 1   
 \end{pmatrix} .
 \end{aligned}
\end{equation}  

In the basis of vectors~\eqref{h}, the action of~$\mathbf B$ is given by the following matrix:
\begin{equation}\label{R}
 \mathrm R = 
 \left(
 \begin{array}{cc|cc|cc|cc}
   1 & 0 & 1 & 0 & 0 & 0 & 1 & 0  \\
   0 & 1 & 0 & 0 & 0 & 1 & 0 & 1  \\
   \hline
   0 & 0 & 1 & 0 & 0 & 0 & 0 & 1  \\
   0 & 1 & 0 & 1 & 0 & 0 & 0 & 0  \\
   \hline
   1 & 0 & 0 & 0 & 1 & 0 & 0 & 0  \\
   0 & 0 & 0 & 0 & 0 & 1 & 1 & 1  \\
   \hline
   1 & 0 & 0 & 1 & 1 & 0 & 0 & 0  \\
   0 & 1 & 0 & 0 & 1 & 0 & 0 & 0  
 \end{array}
 \right) .
\end{equation} 
The exact sense of this is as follows. We consider $\mathrm R$ as a $4\times 4$ matrix whose entries are $2\times 2$ cells (blocks, into which it is divided in~\eqref{R}). Denote, accordingly, $\mathrm R_{jk}$ the cell staying at the intersection of the $j$-th doubled row and the $k$-th doubled column. For instance, $\mathrm R_{34} = \begin{pmatrix} 0 & 0 \\ 1 & 1 \end{pmatrix}$.

Then, for matrix~\eqref{R} and vectors~\eqref{h}, the following relation---analogue of~\eqref{bjk}---holds:
\begin{equation*}
 \begin{pmatrix} \mathsf h_1^{(j)} \\ \mathsf h_2^{(j)} \end{pmatrix} \mathbf B_{jk} = \mathrm R_{jk} \begin{pmatrix} \mathsf h_1^{(k)} \\ \mathsf h_2^{(k)} \end{pmatrix}
\end{equation*}

\label{lastpage}
\end{document}